\documentclass[letterpaper,reqno,11pt]{amsart}

\usepackage{preamble}

\title[Irreducibility of Generalized Permutohedra]{Irreducibility of Generalized Permutohedra, Supermodular Functions, and Balanced Multisets}
\author{Milan Haiman}

\author{Yuan Yao}

\begin{document}

\maketitle

\begin{abstract}
    We study generalized permutohedra and supermodular functions. Specifically we analyze decomposability and irreducibility for these objects and establish some asymptotic behavior. We also study a related problem on irreducibility for multisets.
\end{abstract}

\section{Introduction}\label{sec:introduction}

A \emph{permutohedron} in $\R^n$ is the $(n-1)$-dimensional polytope obtain by taking the convex hull of all $n!$ points obtained by permuting the coordinates a point $(x_1, \dots, x_n)\in \R^n$. In \cite{Postnikov}, Postnikov introduced the \emph{generalized permutohedron}, which is a deformation of a permutohedron obtained by translating the hyperplanes bounding each face. Generalized permutohedra were also further studied in \cite{PRW}. In \cite{Postnikov}, Postnikov derives a volume formula for generalized permutohedra as a polynomial in the defining parameters. The proof of the formula requires a decomposition of some generalized permutohedra into weighted Minkowski sums of coordinate simplices. However, not all generalized permutohedra can be written in this form. Therefore, it is natural to ask whether all generalized permutohedra can be written as a weighted Minkowski sum of a some fixed set of polytopes.

We define a generalized permutohedron to be \emph{irreducible} if it cannot be written as a weighted Minkowski sum of generalized permutohedra in a nontrivial way (any convex polytope is a Minkowski sum of smaller copies of itself). Then all generalized permutohedra can be decomposed as a Minkowski sum of irreducible generalized permutohedra.

In this paper we aim to understand the class of irreducible generalized permutohedra. We make use of connections to related problems.

Generalized permutohedra in $\R^n$ are strongly related to supermodular functions on subsets of $[n]$, which are a discrete analog of convex functions. Supermodular functions are an important object in optimization and other fields (see \cite{Fujishige}). 

There is a direct bijection between irreducible generalized permutohedra and irreducible supermodular functions, which have been studied by several authors (\cite{Submod,PY}). In \cite{PY}, Promislow and Young determined the irreducible supermodular functions for $n\le 4$ and conjectured a simple characterizations for $n>4$. However, this conjecture was shown to be false by {\v{Z}}ivn{\'{y}}, Cohen, and Jeavons in \cite{Submod}, and we further show that this characterization is far from capturing all irreducible supermodular functions.

To understand irreducible supermodular functions, we first study a related (and simpler) problem involving irreducibility. Given a multiset $\M$ of subsets of $[n]$, we say that $\M$ is \emph{balanced} if each element of $[n]$ appears the same number of times in $\M$. We denote the number of times each element appears by the \emph{complexity} $m=m(\M)$. The conditions for irreducibility generalized permutohedra and supermodular functions can be formulated as a modification of that of irreducible balanced multisets.

We derive bounds on the complexity of irreducible balanced multisets and enumerate the number of irreducible balanced multisets up to lower order terms. Using similar ideas, we provide double-exponential upper bounds for complexity and number of irreducible generalized permutohedra. We also obtain double-exponential lower bounds by relating a subclass of supermodular functions to matroids. The key asymptotic results are the following.

\begin{theorem}\label{thm:supermodular-asymptotics}
    The number of irreducible supermodular functions, up to equivalence, is bounded above by $2^{O(n2^n)}$ and bounded below by $2^{\Omega(2^n/n^{3/2})}$.
\end{theorem}

We also study a simple subclass of irreducible supermodular functions, which we enumerate precisely.

The paper is structured as follows. In \cref{sec:prelimaries} we establish the preliminary definitions related to generalized permutohedra and supermodular functions. In \cref{sec:supermodular-analysis} we establish some conditions for a supermodular function to be irreducible. In \cref{sec:balanced} we explore the related problem of irreducible balanced multisets. In \cref{sec:upper-bounds} we obtain upper bounds on the number and complexity of irreducible supermodular functions. In \cref{sec:lower-bounds} we obtain lower bounds on the number of irreducible supermodular functions. In \cref{sec:two-layers}, we study supermodular functions with supermodularities on only two layers.

\section{Preliminaries}\label{sec:prelimaries}

There are several equivalent ways to define a generalized permutohedron, and here we present the one that is the most convenient for our purposes. For a subset $I\subseteq[n]$, we let $1_I$ denote the vector whose $i$-th coordinate is $1$ if $i\in I$ and $0$ otherwise.

\begin{definition}
A \emph{generalized permutohedron} in $\R^n$ is a polytope of the form $$\{x\in \R^n \colon x\cdot 1_I\ge z_I, x\cdot 1_{[n]}=z_{[n]}\}$$ for reals $z_I$ satisfying the \emph{supermodularity condition}: $$z_{I\cap J}+z_{I\cup J}\ge z_I+z_J$$ for all $I,J\subseteq[n]$ (we set $z_\varnothing=0$).
\end{definition}

This definition includes all ordinary permutohedra: for $x_1\le x_2\le \dots \le x_n$, we can recover the permutohedron with vertices that are permutations of $(x_1,\dots,x_n)$ by taking $z_I=x_1+\dots+x_{\abs{I}}$ for each $I\subseteq [n]$.

The supermodularity conditions $z_{I\cap J}+z_{I\cup J}\ge z_I+z_J$ guarantees that a generalized permutohedron has the same face structure as a permutohedron, up to degeneracies that reduce the dimension of faces. In particular, we still have that all edges are parallel to $e_i-e_j$ for some $i,j$.

We define the \emph{Minkowski sum} of two subsets $P$ and $Q$ of $\R^n$ to be the set $P+Q=\{x+y\colon x\in P, y\in Q\}$. Note that the set of generalized permutohedra is closed under Minkowski sums, since taking a Minkowski sum simply results in adding the corresponding $z_I$ parameters.

\begin{definition}
We say that a generalized permutohedron $P$ is \emph{irreducible} if whenever $P$ is written as a Minkowski sum $Q_1+Q_2$ of generalized permutohedra, $Q_1$ and $Q_2$ are both copies of $P$ up to scaling and translation.
\end{definition}

We can view the set of generalized permutohedra as a subset of $\R^{2^n-1}$ by considering the vector of corresponding $z_I$ parameters. This subset is a cone bounded by the hyperplanes corresponding to the supermodularity conditions. Then the irreducible generalized permutohedra correspond to the extreme rays of this cone. In particular, because we have finitely many conditions, there are finitely many irreducible generalized permutohedra, up to scaling and translation. So every generalized permutohedron can be written as a weighted Minkowski sum of irreducible generalized permutohedra. For example, a permutohedron is a weighted Minkowski sum of $\binom{n}{2}$ line segments between the standard basis vectors in $\R^n$.

Since the problem of determining the irreducible generalized permutohedra is equivalent to determining the extreme rays of a high-dimensional cone, we can apply standard algorithms to determine the answer for small $n$. For $n=3$ there are $5$ irreducible supermodular functions. As generalized permutohedra, two of these are equilateral triangles (with opposite orientations), and the other three are are line segments.

For $n=4$ there are $37$ irreducible supermodular functions. Ignoring lower dimensional examples and symmetries, we have $5$ new irreducible generalized permutohedra, pictured below.

\begin{center}
\includegraphics[width=2in]{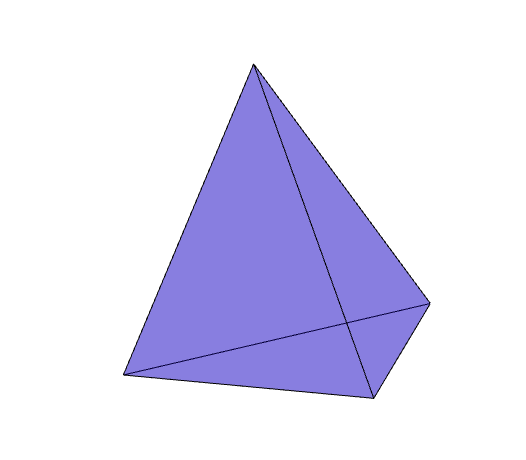}
\includegraphics[width=2in]{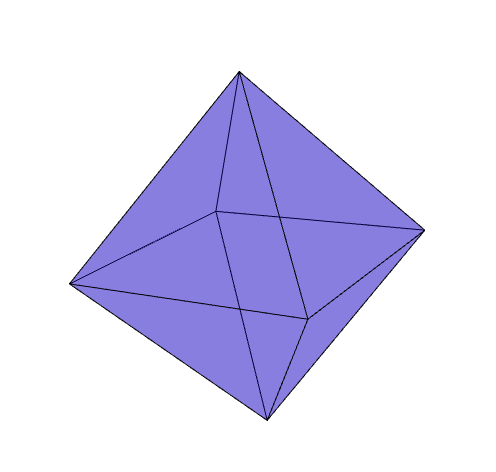}
\includegraphics[width=2in]{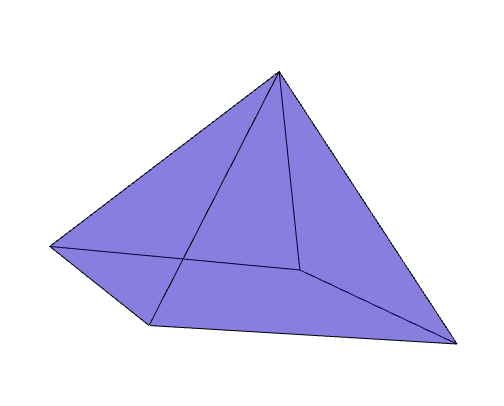}
\includegraphics[width=1.6in]{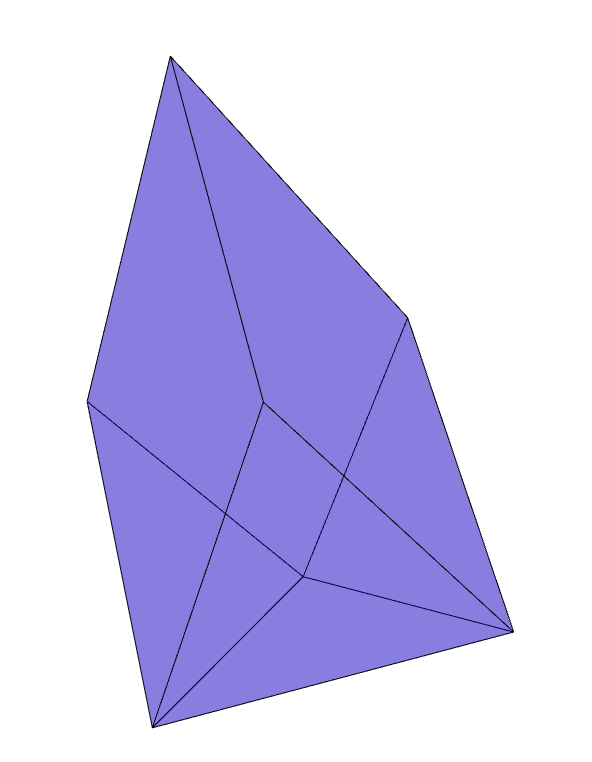}
\includegraphics[width=2in]{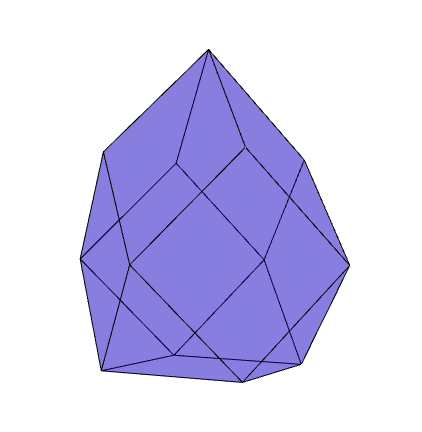}
\end{center}

For $n=5$ there are $117978$ irreducible generalized permutohedra. Even accounting for lower dimensional examples and symmetries, we have many new polytopes that do not follow a clear pattern. Thus we aim to understand the number and complexity of irreducible generalized permutohedra for general $n$ instead of a precise characterization.

\begin{definition}
A function $f: 2^{[n]} \to \R$ is \emph{supermodular} if $f(I\cap J)+f(I\cup J)\ge f(I)+f(J)$ for all $I,J\subseteq[n]$. It is \emph{modular} if $f(I\cap J)+f(I\cup J)= f(I)+f(J)$ for all $I,J\subseteq[n]$.
\end{definition}

Given a generalized permutohedron, we immediately obtain a supermodular function by taking $f(I)=z_I$. In the other direction, any supermodular function $f$ with $f(\emptyset)=0$ gives a generalized permutohedron. Thus generalized permutohedra and supermodular functions are essentially the same object. In particular, note that modular functions correspond to a single point as a generalized permutohedron.

\begin{definition}
We say that two supermodular functions are \emph{equivalent} if they differ by a modular function. We say that a supermodular function $f$ is \emph{irreducible} if it is not modular and whenever $f=g_1+g_2$ for supermodular functions $g_1$ and $g_2$, $g_1$ and $g_2$ are each equivalent to a function of the form $cf$ for $c\in \R_{\ge 0}$.
\end{definition}

Note that each equivalence class of irreducible supermodular functions has a representative $f$ with the following properties:
\begin{itemize}
    \item $f(I)=0$ for $\abs{I}\le 1$
    \item $f$ takes nonnegative integer values with greatest common divisor $1$.
\end{itemize}

As before, we see that the irreducible supermodular functions generate all supermodular functions by taking nonnegative linear combinations. Additionally, irreducible generalized permutohedra correspond to irreducible supermodular functions.


\section{Analyzing Irreducible Supermodular Functions}\label{sec:supermodular-analysis}

As a motivating example, we first consider \emph{nondecreasing} functions on subsets of $[n]$, that is, functions $f\colon 2^{[n]}\to \R$ where $f(I)\ge f(J)$ whenever $I\supseteq J$.

For a set $S$ and $i\in S$, let $\partial_i$ be the \emph{discrete derivative operator} mappings functions $f\colon 2^S\to \R$ to functions $\partial_if\colon 2^{S\setminus\{i\}}\to \R$, defined by $(\partial_i f)(I)=f(I\cup\{i\})-f(I)$. It is clear that $f$ is nondecreasing if and only if $(\partial_i f)(I)\ge 0$ for every $i\in [n]$ and $I\subseteq[n]\setminus\{i\}$. We also have that $f$ is supermodular if and only if $(\partial_i \partial_j f)(I)\ge 0$ for every pair of distinct $i,j\in [n]$ and $I\subseteq[n]\setminus\{i,j\}$. Thus we can think of supermodular functions as having nonnegative second derivatives everywhere, which makes the case of nondecreasing functions (nonnegative first derivatives) natural.

Motivated by this parallel, we can define equivalence and irreducibility for nondecreasing functions.
\begin{definition}
We say two nondecreasing functions are \emph{equivalent} if their difference is a constant function. A nondecreasing function $f$ is \emph{irreducible} if it is not constant and not a nontrivial sum of nondecreasing functions. That is, if $f = g_1 + g_2$ for nondecreasing functions $g_1$ and $g_2$, then $g_1$ and $g_2$ are each equivalent to a nonnegative multiple of $f$.
\end{definition}

We can precisely characterize the irreducible nondecreasing functions. Given an nonempty antichain $\A$ of subsets of $[n]$, we define the up function of $\A$ to be $u_{\A}(I)=1$ if $I\supseteq J$ for some $J\in \A$ and $u_{\A}(I)=0$ otherwise.

\begin{lemma}
A nondecreasing function is irreducible if and only if it is equivalent to a function of the form $cu_{\A}$ for some nonempty antichain $\A$ of subsets of $[n]$ and some $c\in \R_{\ge 0}$.
\end{lemma}

\begin{proof}
First we show the ``if'' direction. Let $\A$ be a nonempty antichain of subsets of $[n]$. We will show that $u_{\A}$ is irreducible. Let $f_1$ and $f_2$ be nondecreasing functions such that $u_{\A}=f_1+f_2$. Now consider sets $I\subseteq J$ such that $u_{\A}(I)=u_{\A}(J)$. We have that $f_i(I)\le f_i(J)$ by nondecreasingity. However $$u_{\A}(I)=f_1(I)+f_2(I)\le f_1(J)+f_2(J)=u_{\A}(J).$$
Thus we must have equality, so $f_1(I)=f_1(J)$ and $f_2(I)=f_2(J)$.

This fact implies that $f_i(I)=f_i([n])$ whenever $I\supseteq A$ for some $A\in \A$, and $f_i(I)=f_i(\emptyset)$ otherwise. Thus $$f_i(I)=f_i(\emptyset)+(f_i([n])-f_i(\emptyset))u_{\A}.$$ This shows that $u_{\A}$ is irreducible.

Now we show the converse. Suppose that $f$ is a nondecreasing function not equivalent to a multiple of $u_{\A}$ for any antichain $\A$. Consider the family $\F$ of subsets $I\subseteq[n]$ for which $f(I)>f(\emptyset)$. Since $f$ is not constant, $\F$ is nonempty. Let $\A$ be the family of minimal sets in $\F$. Note that $u_\A$ takes the value $1$ on elements of $\F$ and $0$ elsewhere. Now, for sufficiently small $c>0$, we have $f(I)>c+f(\emptyset)$. Thus $f-cu_\A$ is nondecreasing for some $c>0$. But $f$ is not equivalent to $cu_\A$. So $f$ is reducible.
\end{proof}

The number of antichains of subsets of $[n]$ is at least $2^{\binom{n}{n/2}}$ by choosing only subsets of size $\binom{n}{n/2}$. In fact such antichains describe most possibilities \cite{antichains-Korshunov,antichains-Pippenger}.

Given this understanding of nondecreasing functions, we can attempt to use it to understand supermodular functions. If $f$ is supermodular, then $\partial_if$ is nondecreasing. So we can construct supermodular functions by taking $n$ nondecreasing functions $g_1,\dots,g_n$ with $g_i\colon 2^{[n]\setminus\{i\}}\to \R$. However, we are restricted by the fact that $$\partial_ig_j=\partial_i\partial_jf=\partial_j\partial_if=\partial_jg_i.$$
Therefore supermodular functions can be heuristically  described as $n$ weighted sums of antichains with a compatibility condition between the sums.

Another way to understand supermodular functions is to consider the supermodularity condition on certain pairs of $I$ and $J$. 
\begin{definition}
We say that an unordered pair of subsets $\{I,J\}$ of $[n]$ is \emph{close} if $\abs{I}=\abs{J}=\abs{I\cap J}+1=\abs{I\cup J}-1$. Let $\mathcal{P}_n$ be the set of all close pairs. Note that $|\mathcal{P}_n| = \binom{n}{2}2^{n-2}$.

Given a supermodular functions $f$, for each close pair $\{I, J\}$, let the \emph{supermodularity value} of this pair be $$s_{I,J}=f(I\cap J)+f(I\cup J)- f(I)-f(J).$$ 
\end{definition}

Clearly, $\{I, J\}$ is a close pair if and only if $1_I$, $1_J$, $1_{I\cap J}$, $1_{I\cup J}$ are the vertices of a square face in the boolean hypercube, so we will treat square faces and close pairs interchangeably.

It is sufficient to define $s_{I,J}$ only when $\{I,J\}$ is a close pair because of the following lemma. Let $T\colon \R^{2^{[n]}}\to \R^{\mathcal{P}_n}$ denote the linear map sending $f$ to $s$.

\begin{lemma}
Let $f\colon 2^{[n]}\to \R$ and let $s=Tf$. Then $f$ is supermodular if and only if $s_{I,J}\ge 0$ for each close pair $\{I,J\}$.
\end{lemma}

\begin{proof}
The ``only if'' direction is clear.

For the ``if'' direction, let $f\colon 2^{[n]}\to \R$ and supppose $s=Tf$ satisfies $s_{I,J}\ge 0$ for each close pair $\{I,J\}$. We will show that $$f(I\cap J)+f(I\cup J)-f(I)-f(J)\ge 0$$ for all $I,J\subseteq [n]$.

Fix subsets $I,J\subseteq[n]$ and let $I\setminus J=\{i_1,\dots,i_{\ell_I}\}$, $J\setminus I=\{j_1,\dots,j_{\ell_J}\}$. Let $K(a,b)=(I\cap J)\cup \{i_1,\dots, i_a,j_1,\dots, j_b\}$ for $0\le a\le \ell_I$ and $0\le b\le \ell_J$. Note that $\{K(a,b-1),K(a-1,b)\}$ is a close pair for $a,b>0$. Additionally we have that $$0\le s_{K(a,b-1),K(a-1,b)}=f(K(a,b))+f(K(a-1,b-1))-f(K(a,b-1))-f(K(a-1,b)).$$
Now we sum inequality over all $1\le a\le \ell_I$ and $1\le b\le \ell_J$. Most terms on the RHS cancel, leaving us with $$0\le f(K(\ell_I,\ell_J))+f(K(0,0))-f(K(\ell_I,0))-f(K(0,\ell_J))=f(I\cup J)+f(I\cap J)-f(I)-f(J).$$
Thus $f$ is supermodular.
\end{proof}

Note that the kernel of $T$ is the space of modular functions, which has dimension $n+1$. Thus the image of $T$ has dimension $2^n-n-1$.

We can determine the image of $T$ in $\R^{\mathcal{P}_n}$ by a set of $\binom{n}{2}2^{n-2}-2^n+n+1$ linear conditions on $s$. The possible vectors $s$ obtained from supermodular functions are just the vectors satisfying these conditions with nonnegative entries. So the relevant subset of vectors is the intersection of $\im T$ with the positive orthant, which is a cone. The irreducible supermodular functions then correspond to the extreme rays of this cone.

We now characterize the linear conditions determining $\im T$. Given a permutation $\sigma=(\sigma_1,\dots,\sigma_n)\in S_n$, we let $I_r(\sigma)=\{\sigma_1,\dots,\sigma_r\}$ and $J_r(\sigma)=\{\sigma_2,\dots,\sigma_{r+1}\}$, for each $1\le r\le n-1$. Also let $I_n(\sigma)=[n]$ and $J_0(\sigma)=\emptyset$, so that $I_r\cup J_r=I_{r+1}$ and $I_r\cap J_r=J_{r-1}$. We define the \emph{path sum} of $s$ along $\sigma$ to be $$P_\sigma(s)=\sum_{r=1}^{n-1}s_{I_r,J_r}.$$
Here $\sigma$ corresponds to a maximal chain in the poset of square faces of the hypercube ordered by the relation $\{I,J\}<\{I',J'\}$ when one of $I'$ and $J'$ contains both of $I$ and $J$ and the other contains at least one of $I$ and $J$.
Additionally, we say that the path corresponding to $\sigma$ has \emph{color} $\sigma_1$. 

\begin{example}

When $n=4$ and $\sigma=(2,4,1,3)$, we obtain the following path on square faces. The color of the path is $\sigma_1=2$, which can be seen by each square face having a pair of opposite edges in the direction $1_{\{2\}}$.

\begin{center}
    \includegraphics[width=3in]{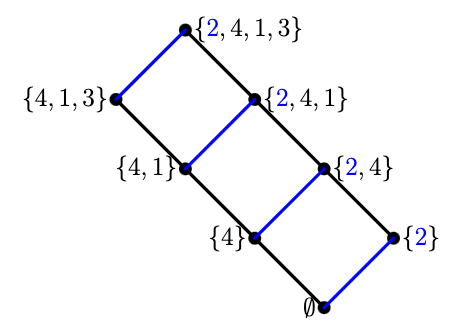}
\end{center}












\end{example}

The following theorem explains the relevance of the color of a path and uses path sums to describe $\im T$.

\begin{theorem}\label{thm:path-sum}
The following are equivalent for any $s\in \R^{\mathcal{P}_n}$:

\begin{enumerate}
    \item $s\in \im T$.
    \item There exist $m_1,\dots,m_n$ such that $P_\sigma(s)=m_{\sigma_1}$ for all $\sigma\in S_n$. The value of $m_i$ will be referred to as the \emph{weight} of color $i$.
    \item For all distinct $i,j,k\in [n]$ and $I\subseteq [n]\setminus\{i,j,k\}$, $$s_{I\cup\{i\},I\cup\{j\}}+s_{I\cup\{i,j\},I\cup\{j,k\}}=s_{I\cup\{i\},I\cup\{k\}}+s_{I\cup\{i,k\},I\cup\{j,k\}}.$$
\end{enumerate}
\end{theorem}

\begin{proof}
We will show that $(1)\implies (2)\implies (3)\implies (1)$.

We first show $(1)\implies (2)$. Let $f$ be a function with $Tf=s$. We claim that  $$m_i=f([n])+f(\emptyset)-f(\{i\})-f([n]\setminus\{i\})=\partial_if([n]\setminus\{i\})-\partial_if(\emptyset)$$ satisfies condition (2). 

Consider an arbitrary $\sigma\in S_n$. Note that $$s_{I_r,J_r}=\partial_{\sigma_1}\partial_{\sigma_{r+1}}f(I_r\cap J_r)=\partial_{\sigma_1}\partial_{\sigma_{r+1}}f(J_{r-1})=\partial_{\sigma_1}f(J_{r})-\partial_{\sigma_1}f(J_{r-1}).$$
Thus the sum in $P_\sigma(s)$ telescopes to $\partial_{\sigma_1}f(J_{n-1})-\partial_{\sigma_1}f(J_0)=m_{\sigma_1},$
as desired.

Next we show that $(2)\implies (3)$. Fix distinct $i,j,k\in [n]$ and $I\subseteq[n]\setminus\{i,j,k\}$. Let $t=\abs{I}$. Choose a $\sigma\in S_n$ such that $\sigma_1=i$, $J_t(\sigma)=I$, $\sigma_{t+2}=j$, and $\sigma_{t+3}=k$. Let $\sigma'\in S_n$ be such that $\sigma_r'=\sigma_r$ for $r\ne t+2,t+3$ and $\sigma_{t+2}'=k$, $\sigma_{t+3}'=j.$ Then we have that $I_r(\sigma)=I_r(\sigma')$ except when $r=t+2$ and $I_r(\sigma)=I_r(\sigma')$ except when $r=t+1$. Since $\sigma_1=\sigma_1'$, we have $P_\sigma(s)=P_{\sigma'}(s)$. Cancelling the common terms from the sum gives
$$s_{I_{t+1}(\sigma),J_{t+1}(\sigma)}+s_{I_{t+2}(\sigma),J_{t+2}(\sigma)}=s_{I_{t+1}(\sigma'),J_{t+1}(\sigma')}+s_{I_{t+2}(\sigma'),J_{t+2}(\sigma')}.$$
After substituting for $I_r$ and $J_r$ we obtain
$$s_{I\cup\{i\},I\cup\{j\}}+s_{I\cup\{i,j\},I\cup\{j,k\}}=s_{I\cup\{i\},I\cup\{k\}}+s_{I\cup\{i,k\},I\cup\{j,k\}},$$ as desired. 

Finally we show $(3)\implies (1)$. Let $s$ be a function satisfying $(3)$. We will construct an $f$ such that $Tf=s$. To construct $f$ we define $f(J)$ inductively based on $\abs{J}$. If $\abs{J}<2$ we let $f(J)=0$.

Now suppose that we have defined $f(J)$ for all $J$ with $\abs{J}<t$, for some $t\in [2,n]$. Fix a $J$ with $\abs{J}=t$. Choose $i,j\in J$ arbitrarily let $I=J\setminus\{i,j\}$. We define $$f(J)=s_{I\cup\{i\},I\cup\{j\}}+f(I\cup \{i\})+f(I\cup \{j\})-f(I).$$
This inductive procedure defines some function $f\colon 2^{[n]}\to \R$. We claim that the choices of $i,j\in J$ do not affect the function $f$ defined by the procedure.

We prove that $f(J)$ is uniquely determined by induction on $\abs{J}$. This is clear for $\abs{J}\le 2$. Now suppose we know that $f(J)$ is uniquely determined for all $J$ with $\abs{J}\le t$ for some $J\in [2,n]$. Fix a $J$ with $\abs{J}=t$. It suffices to show that for any distinct $i,j,k\in J$ we obtain the same value for $f(J)$ be recursing with $\{i,j\}$ or $\{i,k\}$, because applying this fact twice connects any two pairs $\{i,j\}$ and $\{i',j'\}$.

Applying condition $(3)$ with $i,j,k$ and $I=J\setminus\{i,j,k\}$ we obtain that $$s_{I\cup\{i\},I\cup\{j\}}+s_{I\cup\{i,j\},I\cup\{j,k\}}=s_{I\cup\{i\},I\cup\{k\}}+s_{I\cup\{i,k\},I\cup\{j,k\}}.$$
By the inductive hypothesis we have that $$s_{I\cup\{i\},I\cup\{j\}}=f(I\cup \{i,j\})+f(I)-f(I\cup \{i\})-f(I\cup \{j\}).$$
Similarly, we have that $$s_{I\cup\{i\},I\cup\{k\}}=f(I\cup \{i,k\})+f(I)-f(I\cup \{i\})-f(I\cup \{k\}).$$
Substituting these values into condition (3) gives $$s_{I\cup\{i,j\},I\cup\{j,k\}}+f(I\cup \{i,j\})-f(I\cup\{j\})=s_{I\cup\{i,k\},I\cup\{j,k\}}+f(I\cup \{i,k\})-f(I\cup\{k\}).$$
Adding $f(I\cup\{j,k\})$ to both sides gives that the two potential values for $f(J)$ in question are in fact equal. 
\end{proof}

By the above theorem, we know that we can describe $\im T$ using linear conditions of the form $s\cdot v=0$, where $v$ has all entries $0$ except for $2$ entries of $+1$ and $2$ entries of $-1$. We will use this fact to understand the complexity of irreducible supermodular functions.

\section{Irreducibility of Balanced Multisets}\label{sec:balanced}

For each square face of the hypercube, we have an associated supermodularity value $s_{I,J}$. Additionally, we have some subset of our $n!$ paths passing through this face. Our only condition on the supermodularity values is that their sum along each path of a given ``color'' is fixed.

So, choosing a supermodular functions is equivalent to choosing a weight for each square face (and thus the corresponding set of paths) such that the each path of a given ``color'' has a fixed total weight. This is equivalent to choosing a collection of subsets of a set of size $n!$ subject to the sum of the collection having a nice form, which is the same as our simplified irreducibility problem with two modifications. First, we are only allowed to use certain subsets in our collection (i.e., those corresponding to a square face). Second, the sum of the collection doesn't have to be a perfect multiple of the set of all $n!$ paths; it only has to be count paths of each color the same number of times.

\begin{definition}
A multiset $\M$ of subsets of $[N]$ is \emph{balanced} with \emph{complexity} $m$ if each $i\in [N]$ appears in exactly $m$ sets in $\M$. We say that a balanced multiset is \emph{$\Z$-irreducible} if no proper nonempty subset is balanced. 
\end{definition}

\begin{example}
When $N=4$, the multiset $\M=\{\{1\},\{1\},\{2,3\},\{2,4\},\{3,4\}\}$ is balanced with complexity $2$ and is $\Z$-irreducible.
\end{example}

Given a multiset $\M$ of subsets of $[N]$, we can construct a vector $v=v(\M)\in \R^{2^{[N]}}$ such that $v_I$ is the number of times $I$ appears in $\M$. Then we have that $\M$ is balanced (of complexity $m$) if and only if $B_Nv(\M)=m1_{[N]}$, where $B_N$ is the $N\times 2^N$ matrix with columns $1_I$ for each $I\subseteq[N]$. This allows us to extend the definition of balanced multisets to all vectors $v\in \R^{2^{[N]}}$ with nonnegative entries. 

\begin{definition}
A vector $v\in \R_{\ge0}^{2^{[N]}}$ is \emph{balanced} if $B_Nv=m1_{[N]}$ for some $m\in \R$. A balanced vector $v$ is \emph{irreducible} if whenever $v=u_1+u_2$ for balanced $u_1$ and $u_2$, both $u_1$ and $u_2$ are real multiples of $v$. Equivalently, $v$ is irreducible if it lies on an extreme ray of the cone of all balanced vectors in $\R^{2^{[N]}}$.
\end{definition}

Given a nonzero balanced vector $v\in \R^{2^{[N]}}$ we can construct a balanced multiset $\M$ by scaling $v$ to have integer entries not sharing a common factor. We define the complexity of a balanced $v$ to be the complexity of the multiset $\M$ obtained in this way. We also say that a multiset $\M$ is \emph{irreducible} if it is obtained from an irreducible $v$ in this way. 

Note that if $\M$ is irreducible, then it is also $\Z$-irreducible. However, the reverse does not hold. For example, the multiset $\M=\{1234,4,12,135,235,45\}$ is $\Z$-irreducible but not irreducible by \cref{lem:indep}.

To analyze irreducibility for balanced multisets and vectors, we will use some results from random matrix theory. Let $M_N$ be the $N\times N$ matrix with uniform and independent $\pm 1$ entries. In \cite{matrix-singular}, Tikhomirov showed that $M_N$ is singular with probability $(1/2+o(1))^N$. We will only need that $M_N$ is invertible with probability $1-o(1/N)$

By Hadamard's inequality, $\abs{\det M_N}\le n^{n/2}$. Equality is attained when $M_N$ is a Hadamard matrix. Additionally, in \cite{tao-vu}, Tao and Vu showed that $\abs{\det M_N}\ge (cn)^{n/2}$ with probability $1-o(1)$ for fixed $c<1/e$.

For our applications, we will need the following lemma, which follows by applying row operations to $M_N$.

\begin{lemma}\label{lem:row-operations}
Let $A$ be a uniformly random $N\times N$ matrix with $\{0,1\}$ entries and let $A_i$ be $A$ with the $i$-th column replaced with all $1$'s. Then the distribution of $\abs{\det A}$ is $2^{-N}$ times the distribution of $\abs{\det M_{N+1}}$ and the distribution of $\abs{\det A_i}$ is $2^{-N+1}$ times the distribution of $\abs{\det M_N}$.
\end{lemma}
\begin{proof}
We first show the second claim. From $A_i$, replace each column $j\ne i$ with column $i$ minus twice column $j$ to obtain a matrix $A_i'$. We have that $\det A_i'=(-2)^{N-1}\det A_i$, and $A_i'$ has $\pm 1$ entries. Next independently negate each row of $A_i'$ with probability $1/2$ to obtain $A_i''$. Then $\abs{\det A_i''}=\abs{\det A_i'}$ and $A_i''$ is distributed identically to $M_N$. So the distribution of $\abs{\det A_i}$ is $2^{-N+1}$ times the distribution of $\abs{\det M_N}$.

Now we show the first claim. Consider the $(N+1)\times (N+1)$ matrix $A'$ with $A$ in the bottom right $N\times N$ block, all $1$'s in the first column, and all $0$'s in the rest of the top row. Then $\det A'=\det A$. Now, we construct a new matrix $A''$, obtained from $A'$ as follows. Independently, for each $i>1$, with probability $1/2$, either keep column $i$ the same or replace column $i$ with column $1$ minus column $i$. Then $\abs{\det A''}=\abs{\det A'}$ and $A''$ is distributed the same as $A_1$ in the second claim with $N$ replaced by $N+1$. So the distribution of $\abs{\det A}$ is $2^{-N}$ times the distribution of $\abs{\det M_{N+1}}$.
\end{proof}


Now we analyze irreducible balanced vectors. First we prove the following lemma.

\begin{lemma}\label{lem:indep}
Let $v$ be an irreducible balanced vector in $\R^{2^{[N]}}$. Then the set of vectors $\{1_I\colon v_I>0\}$ is linearly independent. 
\end{lemma}

\begin{proof}
Suppose this is not the case. Let $\{I_1,\dots,I_\ell\}=\{1_I\colon v_I>0\}$. Then there exists $\alpha_i$ not all $0$ such that $\sum_{i=1}^\ell \alpha_i1_{I_i}=0.$

Now, let $$t=\min\left\{\frac{v_{I_i}}{\abs{\alpha_i}}\colon 1\le i\le \ell, \alpha_i\ne0\right\}.$$

Define $v^+$ by $v_{I_i}^+=v_{I_i}+t\alpha_i$ and $v_I^+=0$ if $v_I=0$. Similarly define $v^-$ by $v_{I_i}^-=v_{I_i}-t\alpha_i$ and $v_I^-=0$ if $v_I=0$.

By definition of $t$, $v^+$ and $v^-$ have nonnegative entries and $v^++v^-=2v$. Additionally, there exists $i$ such that $v_{I_i}^+=0$ or $v_{I_i}^-=0$ but $v_{I_i}\ne0$. Thus $v^+$ and $v^-$ cannot both be real multiples of $v$. So $v$ is reducible.
\end{proof}

We are now ready to bound the complexity of irreducible balanced vectors. The key idea is to consider the support of such a vector and think about solving for the values of the nonzero entries.

\begin{theorem}\label{thm:balanced-vector-complexity}
Let $v$ be an irreducible balanced vector in $\R^{2^{[N]}}$ with complexity $m$. Then $$m\le \max_{A\in \{0,1\}^{N\times N}} \det A\le (N+1)^{(N+1)/2}/2^N.$$
\end{theorem}

\begin{proof}
Without loss of generality scale $v$ so that it has relatively prime integer entries. In particular we have $B_Nv=m1_{[N]}$.

Let $\{I_1,\dots,I_\ell\}=\{1_I\colon v_I>0\}$. Note that $\ell\le N$ by \cref{lem:indep}. If $\ell<N$ choose $N-\ell$ more sets $J_{\ell+1},\dots,J_N$ such that the set of vectors $\{1_{I_1},\dots,1_{I_\ell},1_{J_\ell+1},\dots,1_{J_N}\}$ is linearly independent.

Let $x$ be the $N\times1$ column vector with $i$-th entry $v_{I_i}$ for $i\le \ell$ and all other entries $0$. Let $A$ be the $N\times N$ matrix with $i$-th column $v_{I_i}$ for $i\le \ell$ and $i$-th column $v_{J_i}$ for $i>\ell$. Then we have that $$Ax=m1_{[N]}.$$
Now, consider solving for $x$ by treating this as a linear system in the entries of $x$. By Cramer's rule, we have that $v_{I_i}=x_i=m\det(A_i)/\det(A)$, where $A_i$ is the matrix $A$ with the $i$-th column replacing by $1_{[N]}$. Since the values $x_i$ are positive integers that are collectively relatively prime, we must have that $$m=\frac{\det(A)}{\gcd(\det(A_1),\dots,\det(A_\ell),\det(A))}\le \det(A).$$
Now by Hadamard's inequality and \cref{lem:row-operations}, we have that $m\le (N+1)^{(N+1)/2}/2^N$, as desired.
\end{proof}

We can extend this result to $\Z$-irreducibility using Carath\'{e}odory's Theorem, losing a negligible factor of $2^N$.

\begin{theorem}
Let $\M$ be a $\Z$-irreducible multiset with complexity $m$. Then $m\le (N+1)^{(N+1)/2}$.
\end{theorem}

\begin{proof}
Let $v=v(\M)$. By Carath\'{e}odory's Theorem, we can write $v$ as a positive linear combination of irreducible balanced vectors $v_1,\dots,v_r$ with $r\le 2^N$. Without loss of generality, each $v_j$ has relatively prime integer entries. 

Now, let $v=\sum_{j=1}^r\lambda_jv_j$ for some reals $\lambda_j\ge 0$. Note that $\lambda_j<1$ because otherwise we could reduce $\M$ from $v=v_j+(v-v_j)$.

Now we have that $$m(\M)1_{[N]}=m(v)1_{[N]}=B_Nv=B_N\sum_{j=1}^r\lambda_jv_j=\sum_{j=1}^r\lambda_jB_Nv_j=\sum_{j=1}^r\lambda_jm(v_j)1_{[N]}.$$

So, by \cref{thm:balanced-vector-complexity} we have that $$m(v)=\sum_{j=1}^r\lambda_jm(v_j)<r(N+1)^{(N+1)/2}/2^N\le (N+1)^{(N+1)/2}.$$

\end{proof}

To construct a lower bound on the largest possible complexity $m$, it suffices to construct a matrix $A$ with large determinant such that the gcd factor is small. It is unlikely that all matrices $A$ with large determinant also give a large gcd factor, but we have not established this yet. We expect that a lower bound is possible losing at most a factor of $c^N$ for some $c\in \R$.

\begin{theorem}
The number of distinct irreducible balanced vectors is $(1-o(1))2^{N^2}/N!$.
\end{theorem}

\begin{proof}

Recall that each irreducible balanced vector is supported on most $N$ distinct sets. For each choice of $N$ distinct sets, we can solve for the unique irreducible balanced vector supported on a subset those sets. So we have an upper bound of $$\binom{2^N}{N}=(1-o(1))\frac{2^{N^2}}{N!}.$$

Now, consider sampling a matrix $A$ from all matrices $N\times N$ matrices with $\{0,1\}$ entries, uniformly at random. Then construct matrices $A_i$ for $1\le i\le N$ by replacing the $i$-th column of $A$ with $\vec{1}_N$.

Since $M_N$ and $M_{N+1}$ are singular with probability $o(1/N)$, by \cref{lem:row-operations}, we have that $A,A_1,\dots,A_N$ are each singular with probability at most $(\frac{1}{2}+o(1))^{-N+1}$. So by a union bound we have that all of $A,A_1,\dots,A_N$ are invertible with probability $1-o(1)$. Now we focus on the $(1-o(1))2^{N^2}$ matrices $A$ for which this holds.

When we solve $Ax=\vec{1}_N$, all entries of $x$ will be nonzero. This gives us a irreducible sequence with exactly $n$ distinct sets. Over all matrices $A$, we will obtain each such sequence exactly $N!$ times, giving a lower bound of $(1-o(1))\frac{2^{N^2}}{N!}$
\end{proof}

This also shows that almost all irreducible sequences have exactly $N$ distinct sets. Also note that we can obtain a smaller error term by using the full strength of Tikhomirov's result in \cite{matrix-singular}.

Obtaining an upper bound for the number of $\Z$-irreducible sequences is more difficult. We can obtain a weak bound by counting all possible sequences with $m\le (N+1)^{(N+1)/2}$ or via Carath\'{e}odory's Theorem, but this unlikely to be tight.

\section{Upper Bounds for Irreducible Supermodular Functions}\label{sec:upper-bounds}

In this section we analyze irreducible supermodular functions. We consider the vector of supermodularity values $s\in \R^{\mathcal{P}_n}$, indexed by square faces of the hypercube. We have that $s$ lies in a fixed subspace of dimension $2^n-n-1$ given the restrictions from \cref{sec:supermodular-analysis}. For $s$ to be irreducible, it must lie on an extreme ray of the cone $s\ge 0$ in this subspace. In particular, any irreducible $s$ must satisfy $2^n-n-2$ linearly independent conditions of the form $s_{I,J}=0$. These conditions allow us to solve for $s$ up to a scaling factor. So, we obtain the following theorem.

\begin{theorem}
There are at most $\binom{n^2 2^n}{2^n} $ irreducible supermodular functions up to equivalence.
\end{theorem}

\begin{proof}
For each irreducible $s$, take $2^n-n-2$ linearly independent conditions of the form $s_{I,J}=0$ that determine $s$. Thus the number of possible irreducible $s$ is at most $$\binom{\binom{n}{2}2^{n-2}}{2^n-n-2}\le \binom{n^2 2^n}{2^n}$$
\end{proof}

This bound is clearly not tight. However, as we see in \cref{sec:lower-bounds} the true growth rate is in fact exponential in $2^n$.

As in the previous section, we can also obtain a bound on complexity. In this case, we define the \emph{complexity} of an irreducible supermodular $f$ as the maximum weight of a color of $s = Tf$ (defined in \cref{thm:path-sum}) after scaling $s$ to have relatively prime integer entries.

\begin{theorem}
Let $f$ be an irreducible supermodular function. Then the complexity of $f$ is at most $2^{n^22^n}$.
\end{theorem}

\begin{proof}
As in the proof of \cref{thm:balanced-vector-complexity}, we consider solving for $s$. Pick $2^n-n-2$ pairs ${I,J}$ with $s_{I,J}=0$ which allow us to determine $s$. Now consider the matrix $A$ in which the top $\binom{n}{2}2^{n-2}-(2^n-n-2)$ rows determine the subspace $\im T$, the next $2^n-n-2$ rows enforce $s_{I,J}=0$, and the last row has all $1$'s to account for scaling. For an irreducible $s$ with integer entries, we have that $As$ is a vector with all $0$ entries except the last entry.

As in the proof of \cref{thm:balanced-vector-complexity}, it suffices to bound the determinant of $A$. The first $\binom{n}{2}2^{n-2}-(2^n-n-2)$ rows each have norm $2$. The next $2^n-n-2$ rows have norm $1$. The last row has norm at most $n2^{n/2}$. So by Hadamard's inequality we have that the determinant of $A$ is at most $$2^{\binom{n}{2}2^{n-2}-(2^n-n-2)}n2^{n/2}\le 2^{n^22^n}.$$
\end{proof}


\section{Lower bounds for irreducible supermodular functions}\label{sec:lower-bounds}

In this section we prove a double-exponential lower bound on the number of irreducible supermodular functions. We do this by relating a special class of supermodular functions to matroids.

\begin{definition}
We say that a supermodular function $f$ is \emph{simple} if the nondecreasing function $\partial_if$ is irreducible or constant for each $i\in[n]$.
\end{definition}

\begin{lemma}\label{lem:01-irreducibility}
    Let $f$ be a simple supermodular function. If $f$ is not irreducible, then there exists a partition $[n]=S_1\cup S_2$ and simple functions $g_1\colon 2^{S_1}\to\R$ and $g_2\colon 2^{S_2}\to\R$ such that $f(I)=g_1(I\cap S_1)+g_2(I\cap S_2)$ for all $I$.
\end{lemma}

\begin{proof}
    Suppose that $f=g_1+g_2$. WLOG each of $f,g_1,g_2$ are standard. For each $i$, $\partial_if=\partial_ig_1+\partial_ig_2$. Since $\partial_if$ is irreducible as a nondecreasing function, one of $\partial_ig_1$ and $\partial_ig_2$ is $0$. Then we can take $S_1=\{i\colon \partial_ig_2=0\}$ and $S_2=[n]\setminus S_1$.
\end{proof}

Recall that a \emph{matroid} $M$ consists of a ground set of elements $E$ and a collection of bases $\B\subseteq 2^E$ satisfying the following exchange axiom: for each pair of bases $A,B\in\B$ and $a\in A\setminus B$, there exists $b\in B\setminus A$ such that $A\cup\{b\}\setminus\{a\}\in\B$.

We will work with matroids $M$ on the ground set $E=[n]$. Recall that each matroid has a rank $r$ which is the size of every base. Additionally, we can define a rank function on subsets $I$ of $E$ by $\rank(I)=\min_{B\in\B}\abs{I\cap B}$. We also have the nullity function given by $\nullity(I)=\abs{I}-\rank(I)$. We say that $I$ is an independent set if $\nullity(I)=0$.

Recall that a loop is an element of $E$ that is in no bases, and a coloop is an element that is in every base. We also say that a matroid $M$ is \emph{reducible} is we can partition the ground set $E$ into two nonempty sets $E_1$ and $E_2$ and construct matroids $M_i$ on $E_i$ with bases $\B_i$ such that $$\B=\{B_1\cup B_2\colon B_1\in \B_1, B_2\in \B_2\}.$$
If there does not exist such a decomposition of $M$, we say that $M$ is \emph{irreducible}. Note that if a matroid on a ground set of at least $2$ elements has a loop or a coloop, then it is reducible.

Next, we define a polytope associated with a matroid. Given a matroid $M$ on $[n]$, its matroid polytope is the convex hull of the indicator vectors for its bases. We use the following result of Gelfand, Goresky, MacPherson, and Serganov \cite{matroid-polytope-equivalence}.

\begin{theorem}[Gelfand, Goresky, MacPherson, and Serganov]\label{thm:matroid-polytope}
Let $P$ be a polytope in $\R^n$ with vertices in $\{0,1\}^n$ such that each edge is a translate of $e_i-e_j$ for some $i,j\in[n]$. Then $P$ is the matroid polytope of some matroid $M$ on $[n]$.
\end{theorem}

\begin{theorem}
    There exists a bijection between equivalence classes of simple supermodular functions and loopless matroids on $[n]$. Furthermore, irreducible functions correspond to irreducible matroids.

\end{theorem}

\begin{proof}
    Let $M$ be a matroid on $[n]$. Then we can construct a supermodular function $f$ by defining $f(I)=\nullity(I)$. Since $M$ is loopless, we have $\nullity(I)=0$ for $\abs{I}\le1$. So, $f$ is the standard representative for its equivalence class. Additionally $f$ is simple since for any $i\in [n]$, $\partial_if$ takes values in $\{0,1\}$.
    
    Now suppose we have a simple function $f$ that is the standard representative for its equivalence class. We will construct a matroid $M$ on $[n]$ with $f(I)=\nullity(I)$. Consider the generalized permutohedron $P$ corresponding to $f$. Let $x$ be a vertex of $P$. Since $f$ takes integer values, $x$ must have integer entries (otherwise we could perturb $x$ along a line while keeping it inside $P$). We have $x_i\ge f(\{i\})=0$, and $x_i\le f([n])-f([n]\setminus\{i\})\in\{0,1\}$. So, $x_i\in\{0,1\}$.

    
    Now, consider an edge of $P$. It is parallel to $e_i-e_j$ for some $i,j\in[n]$. Since the vertices of $P$ are in $\{0,1\}^n$, the edge must be a translate of $e_i-e_j$. Thus by \cref{thm:matroid-polytope} we have that $P$ is the matroid polytope of a matroid $M$. This gives that $M$ maps to $f$ by our map above.
    
    By \cref{lem:01-irreducibility}, a matroid $M$ is irreducible exactly when the corresponding supermodular function $f$ is irreducible.
\end{proof}

Let $m_n$ be the number of matroids on $[n]$. The following bounds for the asymptotics of $m_n$ are known (all logarithms are in base $2$).

\begin{theorem}[Bansal--Pendavingh--van der Pol \cite{matroids-upper-bound}]\label{thm:matroids-upper}
$$\log \log m_n \le n - \frac{3}{2} \log n + \frac{1}{2}\log\frac{2}{\pi}+1+o(1).$$
\end{theorem}

\begin{theorem}[Knuth \cite{matroids-lower-bound}]\label{thm:matroids-lower}
$$\log \log m_n \ge n - \frac{3}{2} \log n + \frac{1}{2}\log\frac{2}{\pi}-o(1).$$
\end{theorem}

\begin{theorem}[Mayhew--Newman--Welsh--Whittle \cite{matroids-loopless}]\label{thm:matroids-loop-coloop}
The number of matroids on $[n]$ with a loop or a coloop is $o(m_n)$.
\end{theorem}

Using these results, we easily obtain the following.

\begin{theorem}
    There are at least $2^{\left(\sqrt{2/\pi}-o(1)\right)2^n/n^{3/2}}$ irreducible supermodular functions.
\end{theorem}

\begin{proof}
    It suffices to lower bound the number of irreducible matroids. The number of reducible matroids on $[n]$ is at most the number of matroids on $[n]$ with a loop or coloop plus $$\sum_{t=2}^{\floor{n/2}}\binom{n}{t}m_tm_{n-t}\le 2^nm_{n-2}m_{\floor{n/2}}\le o(m_n).$$
    Here we used \cref{thm:matroids-upper} and \cref{thm:matroids-lower}. So, with \cref{thm:matroids-loop-coloop} we have that the number of reducible matroids is $o(m_n)$. Thus the number of irreducible matroids is at least $$(1-o(1))m_n\ge 2^{\left(\sqrt{2/\pi}-o(1)\right)2^n/n^{3/2}}.$$

\end{proof}


\section{Supermodularities on two layers}\label{sec:two-layers}

In this section, we analyze irreducible supermodular functions which are nearly modular, for some notion of ``nearly''. Let $\PP_{n,t}$ denote the set of close pairs $\{I,J\}$ with $\abs{I}=\abs{J}=t$, for $t\in[n-1]$.

Suppose that $f$ is a supermodular function and let $s=Tf$ as in \cref{sec:supermodular-analysis}. Suppose that $s$ is supported only on $\PP_{n,t}$ for some fixed $t\in[n-1]$. Then by iterating condition (3) of \cref{thm:path-sum}, we have that $s_{I,J}$ is constant over $\PP_{n,t}$. In particular, there is only equivalence class of irreducible supermodular functions with supermodularities supported on a single layer. As a generalized permutohedron, this corresponds to the hypersimplex $\Delta_{n,n-t}$. Let $\alpha_{n,t}=\max(0,\abs{I}-t)$ be the corresponding standard supermodular function.

A natural next step is to consider the case when $s$ is supported only on $\PP_{n,t}\cup\PP_{n,t+1}$ for some fixed $t\in[n-2]$. Here we allow the supermodularities to lie on two layers of the hypercube instead of just one. Let $\K$ denote the set of standard irreducible supermodular functions of this form.

The hypersimplices $\Delta_{n,n-t}$ and $\Delta_{n,n-t-1}$ correspond to the elements $\alpha_{n,t}, \alpha_{n,t+1} \in\K$ as seen above. Additionally, we can lift the hypersimplex $\Delta_{n-1,n-t+1}$ in $n$ different ways to obtain a supermodular function in $\K$. Specifically, for each $k\in [n]$, the corresponding function is $$\beta_{n,t,k}(I)=\max(0,\abs{I\cap([n]\setminus\{k\})}-t).$$

We will use the $n+2$ functions $\B_{n,t}=\{\alpha_{n,t},\alpha_{n,t+1},\beta_{n,t,1},\dots,\beta_{n,t,n}\}$ to describe $\K$.

\begin{theorem}\label{thm:2-layer-classification}
    The elements of $\K$ other than $\alpha_{n,t}$ and $\alpha_{n,t+1}$ are in bijection with subsets $S\subseteq[n]$ with $\abs{S}\in \{1,n-1\} $ or $$ \min(t+1,n-t)<\abs{S}<\max(t+1,n-t).$$
    
    The bijection is given by the map $$S\mapsto \sum_{k\in S}\beta_{n,t,k}-\max(0,\abs{S}-(t+1))\alpha_{n,t}-\max(0,\abs{S}-(n-t))\alpha_{n,t+1}.$$
\end{theorem}

\begin{proof}
First, notice the identity \begin{equation}\label{eqn:2-layer-identity}
    \sum_{k\in [n]}\beta_{n,t,k}=(n-t-1)\alpha_{n,t}+t\alpha_{n,t+1}.
\end{equation}
Additionally, this is the only linear dependence in $\B_{n,t}$. Thus $\dim\Span\K\ge n+1.$

In fact, we claim that this is an equality. By \cref{thm:path-sum}, for any $f\in \K$, we can solve for $s=Tf$ from the color path sums $m_1,\dots,m_n$ and any fixed supermodularity value on $\PP_{n,t}$, since each path sum has only two nontrivial terms. Since $f$ is standard, we can solve for $f$ from $s$. Thus $\dim\Span\K=n+1.$

Now, let $f\in\K$ with $f\ne \alpha_{n,t},\alpha_{n,t+1}$ and let $s=Tf$. By \cref{eqn:2-layer-identity}, we can write $f$ as a linear combination of $\B_{n,t}$ such that each $\beta_{n,t,k}$ has a nonnegative coefficient $x_k$ and at least one of these coefficients is $0$. Now, note that there must exist a close pair $\{I,J\}\in\PP_{n,t}$ with $s_{I,J}=0$. Otherwise, we could subtract a multiple of $\alpha_{n,t}$ from $f$ while leaving a supermodular function. Thus the coefficient of $\alpha_{n,t}$ is determined by the coefficients $x_k$; it is the minimum supermodularity value of $\sum_{k\in [n]}x_k\beta_{n,t,k}$ on $\PP_{n,t}$. Note that $\beta_{n,t,k}$ is supermodular on a close pair $\{I,J\}\in\PP_{n,t}$ if and only if $k\not\in I\cup J$. So we obtain that the coefficient of $\alpha_{n,t}$ is $$y_1=-\min_{\{I,J\}\in \PP_{n,t}}\sum_{k\not\in I\cup J}x_k=-\min_{\abs{K}=t+1}\sum_{k\not\in K}x_k.$$
Similarly, we obtain that the coefficient of $\alpha_{n,t+1}$ is $$y_2=-\min_{\{I,J\}\in \PP_{n,t+1}}\sum_{k\in I\cap J}x_k=-\min_{\abs{K}=t}\sum_{k\in K}x_k.$$
Next, we show that the coefficients $x_k$ only have one distinct nonzero value. Consider $y_1$ and $y_2$ as functions of $\x=\{x_k\}$. Note that both are piecewise linear. In particular, given $\x$ and $\x'$ with coordinates sharing a (weak) relative order, we have that $y_1(\x+\x')=y_1(\x)+y_1(\x')$. Now, let $S\subseteq[n]$ consist of all $k$ such that $x_k$ is maximal. Let $x_k'=1$ for $k\in S$ and $x_k'=0$ for $k\not\in S$. Then have that $$g=\left(\sum_{k\in S}\beta_{n,t,k}+y_1(\x')\alpha_{n,t}+y_2(\x')\alpha_{n,t+1}\right)$$ is supermodular, and for sufficiently small $\eps>0$, $f-\eps g$ is supermodular. Thus $f$ must be a multiple of $g$. In particular, we can assume $x_k=x_k'\in\{0,1\}$ for each $k$. Then $y_1=-\max(0,\abs{S}-(t+1))$ and $y_2=-\max(0,\abs{S}-(n-t))$. Thus we have that each $f\in\K$ with $f\ne \alpha_{n,t},\alpha_{n,t+1}$ corresponds to an $S\subseteq[n]$ as claimed. It suffices to check with choices of $S$ give an irreducible $f$. By symmetry, we only need to consider $\abs{S}$.

First, we consider the case $\abs{S}\le \min(t+1,n-t)$. If $\abs{S}=0$, then $f=0$, so it is not irreducible. If $\abs{S}=1$, then $f=\beta_{n,t,k}$ for some $k$, so it is irreducible. Otherwise, note that $y_1=y_2=0$, so we already have a decomposition of $f$. So, $f$ is not irreducible.

Next, we consider the case $\abs{S}\ge \max(t+1,n-t)$. If $\abs{S}=n$, then $f=0$ by \cref{eqn:2-layer-identity}, so it is not irreducible. Now assume $\abs{S}<n$. By the above calculations for $y_1$ and $y_2$, we have that for each $\ell\in[n]$, $$\gamma_{n,t,\ell}=\sum_{k\in[n]\setminus\{\ell\}}\beta_{n,t,k}-(n-t-2)\alpha_{n,t}-(t-1)\alpha_{n,t+1}$$ is supermodular. We claim that $f=\sum_{\ell\in[n]\setminus S}\gamma_\ell.$ This follows by applying \cref{eqn:2-layer-identity} and comparing coefficients. So, if $\abs{S}<n-1$, then $f$ is not irreducible.

It remains to check that $f$ is irreducible if $\abs{S}=n-1$ or $$ \min(t+1,n-t)<\abs{S}<\max(t+1,n-t).$$
Now, since we know all possible elements of $\K$, we just need to show that none of the claimed elements can be decomposed using the other claimed elements. It suffices to check that for each ordered pair of claimed elements, there exists a close pair on which the first is strictly supermodular while the second is modular. This is a simple calculation, and we omit the details.

\end{proof}








\section{Acknowledgements}\label{sec:acknowledgements}

We would like to thank Ashwin Sah and Mehtaab Sawhney for helpful comments on random matrix theory. This research was conducted at SPUR at MIT in Summer 2022. We would like to thank the SPUR Directors David Jerison and Ankur Moitra for organizing SPUR and for helpful conversations throughout the project.

\bibliographystyle{plain}
\bibliography{refs}

\end{document}